\documentclass{article}
\usepackage{amsmath,amssymb,amsthm}
\usepackage{tikz}
\usepackage{tkz-graph}
\usetikzlibrary{arrows,shapes,positioning,scopes}

\theoremstyle{plain}
\newtheorem{theorem}{Theorem}[section]
\newtheorem{lemma}[theorem]{Lemma}
\newtheorem{problem}[theorem]{Problem}
\newtheorem{corollary}[theorem]{Corollary}

\theoremstyle{definition}
\newtheorem{definition}[theorem]{Definition}

\newcommand{\domsquare}
{{\mathbin{\kern.5pt\hbox to0pt{\vrule height3pt width.4pt depth0pt\hss}
\hbox{\vrule width9pt height3.4pt depth-3pt\hss}\kern-.45pt
\hbox {\vrule height3pt width.4pt depth0pt\hss}}\kern.5pt}\ignorespaces}
\newcommand{\ransquare}
{{\mathbin{\kern.5pt\hbox to0pt{\vrule height3.4pt width.4pt depth0pt\hss}
\hbox{\vrule width9pt height.4pt depth0pt\hss}\kern-.4pt
\hbox {\vrule height3.4pt width.4pt depth0pt\hss}}\kern.5pt}\ignorespaces}

\newcommand{\A}{\mathbf{A} }
\renewcommand{\S}{\mathbf{S} }
\newcommand{\G}{\mathcal{G} }
\newcommand{\MA}{\mathbf{M(A)} }
\newcommand{\M}{\mathbf{M} }
\newcommand{\K}{\mathcal{K} }
\newcommand{\irel}{\mathtt{i}}
\newcommand{\id}{\mathsf{e}}
\newcommand{\con}[1]{#1\breve{\ }}
\renewcommand{\top}{\mathsf{T} }
\renewcommand{\div}{0'}
\renewcommand{\le}{\leqslant}

\DeclareMathOperator{\dom}{\operatorname{D} }
\DeclareMathOperator{\ran}{\operatorname{R} }
\newcommand{\Dom}{\mathrel{\domsquare}}
\newcommand{\Ran}{\mathrel{\ransquare}}

\begin{document}

\title{Undecidability of representability for lattice-ordered semigroups and ordered complemented semigroups}
\author{Murray Neuzerling}
\date{}

\AtEndDocument{\bigskip{\footnotesize%
  \textsc{Department of Mathematics and Statistics, La Trobe University, VIC 3086, Australia} \par
  \textit{E-mail address}: \texttt{M.Neuzerling@latrobe.edu.au}
}}

\maketitle

\begin{abstract}
 We prove that the problems of representing a finite ordered complemented semigroup or finite lattice-ordered semigroup as an algebra of binary relations over a finite set are undecidable. In the case that complementation is taken with respect to a universal relation, this result can be extended to infinite representations of ordered complemented semigroups.
\end{abstract}

\section{Introduction}

  Relation algebras are a natural extension of Boolean algebras. In addition to lattice operations $\land$ and $\lor$, complementation $-$, a FALSE or bottom constant $0$ and TRUE or top constant $\top$, the signature of relation algebras includes a binary operation of composition $\cdot$, a unary operation of converse $\con{}$, and an identity constant $\id$. For a detailed introduction we refer the reader to~\cite{Maddux2006}.
  
  Such an algebra is called \emph{representable} if it is isomorphic to an algebra of binary relations in which composition is represented as relational composition, $\land$ as intersection, $\lor$ as union, $0$ as the empty set, $\top$ as an equivalence relation $W$, $-$ as complement with respect to $W$, $\con{}$ as relational converse, and $\id$ as the identity relation. In this article we will consider representability for weaker signatures in which only some of these operations are respected. A similar problem is \emph{finite representability}, in which the base set of representation is to be a finite set.
  
  In 1941 Tarski~\cite{Tarski1941} offered a class of algebras in the signature above along with a finite set of axioms. Tarski~\cite[p.~88]{Tarski1941} along with J\'onsson~\cite{JonssonTarski1948} asked if every model of these axioms was isomorphic to an algebra of binary relations. In 1950, Lyndon~\cite{Lyndon1950} offered an infinite family of non-representable relation algebras that satisfied Tarski's axioms. Monk~\cite{Monk1964} later showed that no finite axiomatisation is possible. Finally in 2001 Hirsch and Hodkinson~\cite{HirschHodkinson2001} showed that representability is undecidable for finite relation algebras. So in the full signature we have non-finite axiomatisability of representability and undecidability of representability for finite algebras.
  
  A natural question then is whether or not these results hold for reducts of the full signature. A survey on this topic is offered by Schein~\cite{Schein1991}. In particular, Schein remarks that ``it would be interesting to describe `complemented semigroups'\dots This problem may be more treatable for ordered complemented semigroups.'' These are algebras with signature $\langle \cdot, \le, - \rangle$. In this article we show that representability in this signature is undecidable if complements are taken with respect to a universal relation. 
  
  Representability over a finite base set is also shown to be undecidable, a result we are able to extend to a weaker notion of complementation. We also prove undecidability of finite representability for lattice-ordered semigroups, which are those with signature $\langle \cdot, \land, \lor \rangle$. Furthermore, these results regarding either representability or finite representability apply to any signature between one of these and that of a Boolean monoid, $\langle \cdot, \land, \lor, -, \id, 0, \top \rangle$.
  
  In order to prove these results we adapt a construction of Boolean monoids used by Hirsch and Jackson~\cite{HirschJackson2012}, correcting issues that arise from weakening the signature. 

\section{Partial groups and Boolean monoids} 

  A \emph{partial group} is a system $\A = \langle A; *, e \rangle$ with a partial binary operation $*$ such that $*$ is associative and $e$ acts as an identity whenever $*$ is defined. That is to say, $\A$ is a group except that some compositions are undefined.
  
  Furthermore, $\A$ is also a \emph{square partial group} if there is a subset $\sqrt{A}$ of $A$ containing the identity $e$ such that the following holds.
\begin{enumerate}
  \item $a * b$ is defined if and only if $a,b \in \sqrt{A}$.
  \item $\sqrt{A} * \sqrt{A} = A$; that is, for every $c \in A$ there are $a,b \in \sqrt{A}$ such that $a * b = c$.
\end{enumerate}

  A partial group $\A$ is \emph{cancellative} if it satisfies the cancellation laws
\begin{align*}
  x * y = x * z &\implies y = z \\
  \text{and } x * y = z * y &\implies x = z.
\end{align*}

  From a finite cancellative square partial group $\A$ the authors of~\cite{HirschJackson2012} construct a finite Boolean monoid $\MA$ with signature $\langle \cdot, \land, \lor, -, \id, 0, \top \rangle$. While we will not go into the details of the construction here, it is worth noting that the resultant $\MA$ is a \emph{normal} Boolean monoid. That is, if $\dom(a) = (a\top) \land \id$ and $\ran(a) = (\top a) \land e$ then $\dom(a)a = a = a\ran(a)$. In a representation of a normal Boolean monoid, $\dom(a)$ and $\ran(a)$ will be represented as a restriction of the identity relation to the domain and range of $a$, respectively. Note also that $\dom(a)$ and $\ran(a)$ are idempotent in $\MA$.
  
  This construction from a partial group references the \emph{partial group embedding problem} for a class of groups $\K$. This problem takes a finite partial group $\A$ and returns YES if there is a group $\G \in \K$ and an injective map $\phi\colon A \to G$ that respects all products defined in $\A$. Evans~\cite{Evans1953} showed that this problem is decidable for a class $\K$ if and only if the uniform word problem for $\K$ is decidable. In particular, this problem is undecidable if $\K$ is either the class of groups or class of finite groups, and $\A$ is a finite cancellative square partial group~\cite[Lemma~3.4]{HirschJackson2012}.

  One of the key concepts required to prove undecidability of representability is a formal means of referring to all elements that act as injective partial maps, hereafter called \emph{injective functions}. In the full signature of relation algebras, one can consider a unary relation $\irel$ as in Definition~\ref{defn:irelbooleanmonoid} to capture these elements. 

\begin{definition}\label{defn:irelbooleanmonoid}
  Define a unary relation $\irel$ in the language of relation algebras by
\[
  x \in \irel \iff x\con{x} \le \id \text{ and } \con{x} x \le \id.
\]
\end{definition}

  In a representation respecting converse, composition and identity, elements in $\irel$ are exactly those relations that would be represented as injective functions. By considering the diversity relation $\div = -\id$, we can also view $\irel$ as the set of elements satisfying the following formula in a signature containing $\{\cdot, \land, \div\}$.

\begin{lemma}[{\cite[Lemma~2.12]{HirschJackson2012}}]\label{lemma:irel}
  Let $\mathbf{R}$ be a relation algebra. Then $a \in \irel^\mathbf{R}$ if and only if
\[
  (a\div) \land a = 0 = (\div a) \land a.
\]
\end{lemma}

  The final concepts needed are those of domain and range equivalence. Binary relations $r$ and $s$ in an algebra over base set $X$ are domain equivalent, denoted $r \Dom s$, if
\[
  \{x \in X \mid (\exists y \in X) (x,y) \in r \} = \{x \in X \mid (\exists y \in X) (x,y) \in s \}.
\]
  We use the same notation for the abstraction of this concept in a Boolean monoid, with $x \Dom y$ if $\dom(x) = \dom(y)$. Range equivalence $\Ran$ is defined similarly. Note that in signatures weaker than that of a Boolean monoid a representation may preserve $\Dom$ or $\Ran$ without necessarily preserving $\dom$ or $\ran$, respectively.

  The following theorem is a combination of Propositions 5.1 and 6.3 from~\cite{HirschJackson2012}.
\begin{theorem}\label{thm:mainundecidabilitytheorem}\label{thm:piggyback}
  Let $\A$ be a finite cancellative square partial group. The following are equivalent, with the statements in square brackets giving a separate set of equivalences.
\begin{enumerate}
  \item $\MA$ is representable [over a finite base set].
  \item There is a $\{ \cdot, \irel, \Dom, \Ran\}$-embedding of $\MA$ into $\wp(X \times X)$ for some [finite] set $X$.
  \item $\A$ embeds into a [finite] group $\G$.
\end{enumerate}
\end{theorem}

  We have already observed that both versions of (3) are known to be undecidable, and so too are the items in (1) and (2). The goal of this paper is to introduce equivalent statements regarding decidability of representability of signatures weaker than that of a Boolean monoid.
  
  In considering signatures without converse, we cannot be certain that $\top$ is represented as an equivalence relation. It turns out that a representable normal Boolean monoid can always be represented in such a way that $\top$ acts as an equivalence relation~\cite[Lemma~2.2]{HirschJackson2012}. This does not necessarily hold for weaker signatures. While this requirement on $\top$ is not without precedent for reducts of relation algebras (see~\cite{Mikulas2004, Schein1991}) we may wish to remove it, or even consider algebras in which no top element exists, and so we will always state when this assumption is in use. With this in mind, we introduce statements in Theorem~\ref{thm:main} equivalent to those in Theorem~\ref{thm:mainundecidabilitytheorem}, but regarding representability of lattice-ordered semigroups and ordered complemented semigroups, thus proving undecidability of these problems as well.

\begin{theorem}\label{thm:main}
  Let $\A$ be a finite, cancellative, square partial group. The following are equivalent.
\begin{enumerate}
  \item $\MA$ is representable [over a finite base set].
  \item $\MA$ is representable [over a finite base set] as a lattice-ordered semigroup with $\top$ represented as an equivalence relation.
  \item $\MA$ is representable [over a finite base set] as an ordered complemented semigroup with $\top$ represented as an equivalence relation.
  \item There is a $\{ \cdot, \irel, \Dom, \Ran\}$-embedding of $\MA$ into $\wp(X \times X)$ for some [finite] set $X$.
  \item $\A$ embeds into a [finite] group $\G$.
\end{enumerate}
\end{theorem}
  We note that a representation of $\MA$ as a lattice-ordered semigroup would respect the operations in $\{\cdot, \land, \lor\}$, while a representation as an ordered complemented semigroup would respect those in $\{\cdot, \le, -\}$. Since both signatures are weaker than that of a Boolean monoid we can see that $(1) \implies (2)$ and $(1) \implies (3)$. Similarly we note that these results apply to any signature between one of these reducts and that of a Boolean monoid. 
  
  The remaining implications are $(2) \implies (4)$ and $(3) \implies (4)$. Since composition is preserved in a representation of a semigroup, this aspect is trivial. We must prove that relations in $\irel$ are preserved as injective functions under a representation in either signature. We do this for lattice-ordered semigroups in Lemma~\ref{lemma:latticeordered} and ordered complemented semigroups in Lemma~\ref{lemma:orderedcomplemented}. In Lemma~\ref{lemma:domranequivalence}, we prove that a representation of $\MA$ in either reduct preserves domain and range equivalence. 
  
  There is some ambiguity here as to the definition of complementation in a reduct. In the full signature of relation algebras complementation is taken with respect to the top element. In the absence of a top element, one can declare that if $x$ is related to $y$ by an element of the algebra, then for all relations $a$ we have that $(x,y)$ belongs to just one of $\{-a,a\}$. This mimics the behaviour of complementation when a top element is present by taking complements with respect to the union of all elements. We call this \emph{relative complementation}. A stronger definition would take complements with respect to a universal relation, demanding that every $(x,y)$ belongs to just one of $\{-a,a\}$. We refer to this as \emph{universal complementation}.
    
  For lattice-ordered semigroups and ordered complemented semigroups with relative complementation, the requirement on $\top$ can be removed if the representation is to be over a finite base set, and we show this in Lemmas~\ref{lemma:latticeorderedtop} and \ref{lemma:orderedcomplementedtop}. A representation of an ordered complemented semigroup with universal complementation will always represent $\top$ as an equivalence relation if it exists, as shown in Lemma~\ref{lemma:strongcomplement}, and so representability of algebras in this signature is undecidable. These results are stated in Theorem~\ref{thm:succinctresults}.

\begin{theorem}\label{thm:succinctresults}
  Let $\tau$ be a signature such that $\tau \subseteq \{\cdot, \land, \lor, -, \id, 0, \top\}$. The following problems are undecidable.
\begin{itemize}
  \item Finite representability of algebras with signature $\tau$ where $\{\cdot, \land, \lor\} \subseteq \tau$.
  \item Finite representability of algebras with signature $\tau$ where $\{\cdot, \le, -\} \subseteq \tau$ and $-$ is to be represented as relative complementation.
  \item Representability and finite representability of algebras with signature $\tau$ where $\{\cdot, \le, -\} \subseteq \tau$ and $-$ is to be represented as universal complementation.
\end{itemize}
\end{theorem}

  Theorem~\ref{thm:succinctresults} also yields results about non-finite axiomatisability for the same signatures. If there exists a finite set of first-order axioms characterising representability of a class of algebras, then one can consider an algorithm that checks a finite algebra against each of these axioms to determine representability. Hence, finite axiomatisability implies decidability of representability, giving us the results in Corollary~\ref{cor:succinctresults}.

\begin{corollary}\label{cor:succinctresults}
  Let $\tau$ be a signature such that $\tau \subseteq \{\cdot, \land, \lor, -, \id, 0, \top\}$. The following classes of algebras are not finitely axiomatisable in first order logic.
\begin{itemize}
  \item Any class whose finite members are the finitely representable algebras with signature $\tau$ where $\{\cdot, \land, \lor\} \subseteq \tau$.
  \item Any class whose finite members are the finitely representable algebras with signature $\tau$ where $\{\cdot, \le, -\} \subseteq \tau$ and $-$ is to be represented as relative complementation.
  \item Representable algebras with signature $\tau$ where $\{\cdot, \le, -\} \subseteq \tau$ and $-$ is to be represented as universal complementation.
\end{itemize}
\end{corollary}

  We note that non-finite axiomatisability of representability of algebras with signature $\tau$ where $\{\cdot, \land, \lor\} \subseteq \tau \subseteq \{\cdot, \land, \lor, \con{}, \id, 0, \top\}$ was shown by Andr\'eka~\cite{Andreka1991}; see also Andr\'eka and Mikul\'{a}s~\cite{AndrekaMikulas2011}. 
  
\section{Proofs of main results}
  
  Let $\S = \langle S; \cdot, \id \rangle$ be a monoid. We also consider meet $\land$, join $\lor$, complementation $-$, a partial order relation $\le$, and constants $0$ and $\top$.

  Let $h\colon \S \to \wp ( X \times X )$ be a representation of $\S$ on a base set $X$ preserving at least composition. Define an equivalence relation $\sim$ on $X$ such that for all $x,y \in X$, $x \sim y$ if $x = y$ or $\id$ acts as the universal relation on the set $\{x,y\}$, a situation illustrated in Figure~\ref{fig:compressed}. Define a new representation $\hat{h} \colon \S \to \wp( X/\!\!\sim \times\, X/\!\!\sim )$ such that for $a \in S$, 
\[ 
  \hat{h} \colon a \mapsto \{ [x], [y] \mid (\exists x' \in [x])(\exists y' \in [y]) \; (x',y') \in h(a) \}.
\]
 
\begin{figure}[ht]
\centering
\begin{tikzpicture}[->,>=stealth',shorten >=1pt,auto,node distance=3cm,thick]]
\tikzstyle{vertex}=[circle, fill=black, draw=black, inner sep = 0.06cm]

\node[vertex, label = below left: $x$] (x) at (0:0cm) {};
\node[vertex, label = above right: $y$] (y) at (45:2cm) {};
\draw[->] (x) .. controls ([xshift=0mm, yshift=1.44cm] x) and ([xshift=-7.22mm, yshift=0cm] y)  .. (y);
\draw[->] (y) .. controls ([xshift=0mm, yshift=-1.44cm] y) and ([xshift=7.22mm, yshift=0cm] x)  .. (x);
\Loop[dist=2cm,dir=SOWE,label=$\id$,labelstyle=below left](x);
\Loop[dist=2cm,dir=NOEA,label=$\id$,labelstyle=above right](y);
\node at (1.3cm,0.1cm) {$\id$};
\node at (0cm,1.3cm) {$\id$};
\end{tikzpicture}
\caption{$\id$ acting as the universal relation on $\{x,y\}$}\label{fig:compressed}
\end{figure}
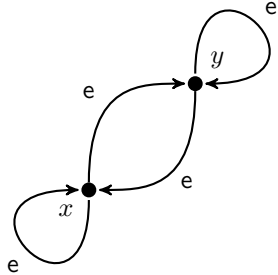

\begin{lemma}\label{lemma:quotient}
  If $h$ is a representation of $\S$ preserving composition then so too is $\hat{h}$. Furthermore, $\hat{h}$ preserves Boolean operations and constants $0$, $\top$ and $\id$, if they are correctly represented by $h$.
\end{lemma}
\begin{proof}
  Consider $a,b \in S$ such that $a \ne b$. Then, since $h$ is faithful, we may assume without loss of generality that there exists $(x,y) \in h(a) \backslash h(b)$. Then $([x], [y]) \in \hat{h}(a)$. Suppose by way of contradiction that $([x],[y]) \in \hat{h}(b)$. Then there exists $(w,z) \in h(b)$ with $\id$ acting as the universal relation on $\{x,w\}$ and on $\{y,z\}$. That is, $(x,w) \in h(\id)$ and $(z,y) \in h(\id)$. Since $h$ preserves composition we have that $(x,y) \in h(ebe)$ and so $(x,y) \in h(b)$. But this violates the assumptions on $(x,y)$. So $\hat{h}$ is faithful.
  
  Now we turn our attention to composition under $\hat{h}$. Let $([x],[y]) \in \hat{h}(a)$ and $([y],[z]) \in \hat{h}(b)$. Without loss of generality, assume $(x,y) \in h(a)$ and $(y,z) \in h(b)$, since as before we can always compose elements with $\id$ to move around within equivalence classes. Then, as $h$ preserves composition, $(x,z) \in h(ab)$ and so $([x],[z]) \in \hat{h}(ab)$. Similarly we have that $([x],[z]) \in \hat{h}(ab) \implies ([x],[z]) \in \hat{h}(a)\hat{h}(b)$. So $\hat{h}$ also preserves composition.

  We note that $\hat{h}$ only contracts binary relations in $h(\S)$. Hence, Boolean operations and constants $0$, $\top$ and $\id$ are preserved in $\hat{h}$, assuming they were correctly represented by $h$. In particular, if $\id$ is represented correctly then $(x,y) \in h(\id) \iff x = y$, and so $h(\id) = \hat{h}(\id)$.
\end{proof}

  It is by this quotient that we will ensure that the elements of the Boolean monoid $\MA$ that are in $\irel$ are represented as injective functions. Recall from Lemma~\ref{lemma:irel} that an element $a \in \irel$ if and only if $(a\div) \land a = 0 = (\div a) \land a$.

\begin{lemma}\label{lemma:latticeordered}
  Suppose that the Boolean monoid $\MA$ is representable in a signature containing $\{\cdot, \land, \lor\}$ in such a way that $\top$ is represented as an equivalence relation. Then there exists a representation in the same signature with the property that if $a \in M(A)$ is such that $(a \div ) \land a = 0 = (\div a) \land a$ then $a$ is represented as an injective function.
\end{lemma}
\begin{proof}
  Let $h \colon \MA \to \wp(X \times X)$ be such a representation of $\MA$ onto some base set $X$ and consider $a \in M(A)$ such that $(a \div ) \land a = 0 = (\div a) \land a$. By applying Lemma~\ref{lemma:quotient} we may work under the assumption that $h = \hat{h}$, and note that this preserves the property that $\top$ is represented as an equivalence relation. 
  
  Suppose there exists $x,y,z \in X$ such that $(y,x) \in h(a)$ and $(y,z) \in h(a)$, a situation illustrated in Figure~\ref{fig:anotfunction}. We will show that $x=z$. As $\top$ is acting as the universal relation on $\{x, y, z\}$ and $0\top = \top 0 = 0$, if $h(0)$ relates any two (potentially equal) elements of $\{x, y, z\}$ then it must act as the universal relation on all three. Since $0 \le e$, this would imply that $e$ is also acting as a universal relation, a situation we have precluded unless $x = z$. So assume otherwise, that is, assume that $h(0)$ is not relating any two elements of $\{x, y, z\}$.
   
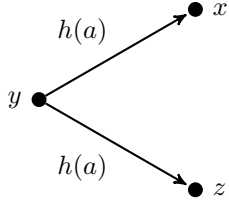
\begin{figure}[ht]
\centering
\begin{tikzpicture}[->,>=stealth',shorten >=1pt,auto,node distance=3cm,thick,scale=0.8]
\tikzstyle{vertex}=[circle, fill=black, draw=black, inner sep = 0.06cm]

\node[vertex, label = left: $y$] (y) at (0:0cm) {};
\node[vertex, label = right: $x$] (x) at (30:3cm) {};
\node[vertex, label = right: $z$] (z) at (-30:3cm) {};

\path[->,font=\normalsize]
(y) edge node [above left] {$h(a)$} (x)
(y) edge node [below left] {$h(a)$} (z);
\end{tikzpicture}
\caption{An element $a$ not represented as a function under $h$.}\label{fig:anotfunction}
\end{figure}

  We note also that $(x,z) \notin h(\div)$ as if this were the case then we would have $(y,z) \in h(a\div)$. But $a \land (a\div) = 0$, giving $(y,z) \in h(0)$. Similarly, we have that $(z,x) \notin h(\div)$. As $\div  \lor \id = \top$ we therefore have $(x,z)$ and $(z,x)$ in $h(e)$. Hence $(x,x)$ and $(z,z)$ are in $e$. As $\hat{h} = h$ it follows that $x = z$ as required. That is, $a$ is a function under $h$. By symmetry we also have that $a$ is injective under $h$.
\end{proof}  
  
  The requirement that $(\div a) \land a = 0 = (a \div) \land a$ simply ensures that $a$ is disjoint from $a \div$ and also from $\div a$. We can restate this with operations in $\{\cdot, \le, -\}$ such that $(\div a) \land a = 0$ if and only if $a \le -(\div a)$, and similarly $(a \div) \land a$ if and only if $a \le -(a\div)$. This allows us to replicate the previous result in the signature of ordered complemented semigroups.

\begin{lemma}\label{lemma:orderedcomplemented}
  Suppose that the Boolean monoid $\MA$ is representable in a signature containing $\{\cdot, \le, -\}$ in such a way that $\top$ is represented as an equivalence relation. Then there exists a representation in the same signature with the property that if $a \in M(A)$ is such that $a \le -(a\div)$ and $a \le -(\div a)$ then $a$ is represented as an injective function.
\end{lemma}
\begin{proof}
  Let $h \colon \MA \to \wp(X \times X)$ be a such a representation of $\MA$ onto some base set $X$ and consider $a \in M(A)$ such that $a \le -(a\div)$ and $a \le -(\div a)$. Again we work under the assumption that $h = \hat{h}$. Since $0$ is the unique element with the property that $0 \le -0$, we have that $h(0) \subseteq h(\top)\backslash h(0)$, and so $0$ is represented correctly as the empty set.
  
  We take $x,y,z$ as in Figure~\ref{fig:anotfunction} with $(y,x) \in h(a)$ and $(y,z) \in h(a)$. As $a \le -(a\div)$, we cannot have $(x,z) \in h(\div )$, since we could compose to get $(y,z) \in h(a \div )$. Similarly, $(z,x) \notin h(\div)$. Because $\id$ and $\div$ are complementary with respect to $\top$, we have that $(x,z)$ and $(z,x)$ are in $h(\id)$. We compose to realise $\id$ acting as a universal relation on $\{x,z\}$, a situation we have precluded unless $x = z$. Hence, $a$ is represented as a function under $h$. By symmetry we also have that $a$ is injective under $h$.
\end{proof}

  Hence, the $\irel$ relation as given in Definition~\ref{defn:irelbooleanmonoid} can be recovered in the case of signatures containing  $\{\cdot, \land, \lor\}$ or $\{\cdot, \le, -\}$, as long as $\top$ is to be represented as an equivalence relation. To complete the $\{\cdot, \irel, \Dom, \Ran\}$-embedding required by Theorem~\ref{thm:main}, we must also check that domain and range equivalence are respected in the representation of a Boolean monoid as either a lattice-ordered semigroup or an ordered complemented semigroup.

\begin{lemma}\label{lemma:domranequivalence}
  Suppose that the Boolean monoid $\MA$ is representable in a signature containing $\{\cdot, \land, \lor\}$ or in a signature containing $\{\cdot, \le, -\}$, and in either case suppose that $\top$ is represented as an equivalence relation. Then one can define domain and range equivalence of the elements in $\MA$ in such a way that they are respected by the representation.
\end{lemma}
\begin{proof}
  In either case, take $x \Dom y$ if $x\top = y\top$, and $x \Ran y$ if $\top x = \top y$.
\end{proof}

  This establishes that $(2) \implies (4)$ and $(3) \implies (4)$ in Theorem~\ref{thm:main}, completing the proof. Subject to the assumption that $\top$ is represented as an equivalence relation, we have undecidability of representability and finite representability of finite algebras in either signature. If we restrict our attention to representations over a finite base set then we can remove this assumption. The proofs are largely the same for lattice-ordered semigroups and ordered complemented semigroups, and both involve Lemma~\ref{lemma:idempotentlemma}. 
  
\begin{lemma}\label{lemma:idempotentlemma}
  Let $h$ be a representation of a Boolean monoid $\M$ onto a finite base set $X$ respecting composition and order. Then for every nonzero idempotent $f \in M$ there exists an element of $X$ fixed by $h(f)$ but not by $h(0)$.
\end{lemma}
\begin{proof}
  By faithfulness there exists $x,y \in X$ such that $(x,y) \in h(f)$ and $(x,y) \notin h(0)$, or such that $(x,y) \notin h(f)$ and $(x,y) \in h(0)$. Since $0$ is the bottom element, we conclude that the latter is not possible and assume that $(x,y) \in h(f)$. Since $f$ is idempotent we must witness an element $z \in X$ such that $(x,z) \in h(f)$ and $(z,y) \in h(f)$. We must continue to witness this for every pair in $h(f)$. But the representation is finite, so we must eventually witness a loop $(x_a,x_a) \in h(f)$. If $(x_a, x_a) \in h(0)$ also, then we could compose to get $(x,y) \in h(0)$, violating our initial assumption. Hence, $f$ but not $0$ fixes $x_a$ in the representation.
\end{proof}

  We first remove the assumption that $\top$ be represented as an equivalence relation in finite representations of $\MA$ as a lattice-ordered semigroup, although in actuality only composition and meet are required for the proof. Recall that in the Boolean monoid signature we defined $\dom(a) = (a\top) \land \id$ and that $\dom(a)$ is idempotent in $\MA$.

\begin{lemma}\label{lemma:latticeorderedtop}
  Let $h$ be a representation of the Boolean monoid $\MA$ onto a finite base set $X$ respecting the operations in $\{\cdot, \land\}$. Then there exists a representation $h^\circ$ in the same signature but representing the top element $\top$ as an equivalence relation. Furthermore, if $h$ respects the operations in $\{\lor,-,\id, 0\}$ then so too does $h^\circ$. 
\end{lemma}
\begin{proof}
  For a binary relation $r$ define the symmetric interior 
\[
  r^\circ := \{ (x,y) \mid (x,y) \in r \text{ and } (y,x) \in r \}.
\]
  If $r$ is reflexive and transitive then one can view $r^\circ$ as the largest equivalence relation contained in $r$. Define $h^\circ \colon M(A) \to \wp(X \times X)$ as $h^\circ \colon a \mapsto h(a) \cap h(\top)^\circ$.
  
  Since we are only omitting non-loops in the representation we have that $h^\circ$ preserves any operation in $\{\lor,-,\id, 0\}$, assuming that $h$ does. For composition, consider $(x,y) \in h^\circ(ab)$ for some $a,b \in M(A)$. Then since $h$ respects composition there exists $z \in X$ such that $(x,z) \in h(a)$ and $(z,y) \in h(b)$. As $(x,y)$ is in the image of $h^\circ$ we have that $(y,x) \in h(T)^\circ$. So $(y,z) \in h(\top a)$ and, as $\top$ is the top element, $(y,z) \in h(\top)$. Similarly, $(z,x) \in h(\top)$. We conclude that $(x,z) \in h^\circ(a)$ and $(z,x) \in h^\circ(b)$. By similar composition with $\top$ we have that if $(x,z) \in h^\circ(a)$ and $(z,y) \in h^\circ(b)$ then $(x,y) \in h^\circ(ab)$, and so composition is respected by $h^\circ$.
  
  Now we must prove that $h^\circ$ is faithful. Let $a,b \in M(A)$ be distinct and assume without loss of generality that $b \nleq a$, so that $b\land(-a) \ne 0$. Note that we are only considering $-a$ as an element of $M(A)$, and do not require complementation to be represented in any way. As $\MA$ is normal we have that $\dom(b \land(-a))(b \land(-a)) = (b \land(-a))$, and so $\dom(b\land(-a)) \ne 0$. We established in Lemma~\ref{lemma:idempotentlemma} that nonzero idempotents under $h$ fix points in $X$ that are not fixed by $0$. As such, $h^\circ(\dom(b\land(-a))) \ne h^\circ(0)$. But clearly $h^\circ(\dom(a\land(-a))) = h^\circ(0)$. As such $h^\circ(a) \ne h^\circ(b)$, and so $h^\circ$ is faithful.

  Since $h$ respects composition and order, $h^\circ$ represents $\top$ as transitive and symmetric, and hence reflexive on a subset of $X$. Since $h^\circ$ is faithful, this subset is nonempty. Hence, $h^\circ$ represents $\top$ as an equivalence relation over a nonempty subset of $X$.
\end{proof}

   If we have a representation of $\MA$ as a lattice-ordered semigroup, then we can take the symmetric interior and then the quotient used in Lemma~\ref{lemma:quotient} to obtain a similar representation preserving the $\irel$ relation and representing $\top$ as an equivalence relation. This allows us to remove the requirement in Lemma~\ref{lemma:latticeordered} that the top element of $\MA$ be represented as an equivalence relation, if the representation is to be taken over a finite set. The following lemma permits us to do the same in the case of Lemma~\ref{lemma:orderedcomplemented}, which deals with ordered complemented semigroups.

\begin{lemma}\label{lemma:orderedcomplementedtop}
  Let $h$ be a representation of the Boolean monoid $\MA$ onto a finite base set $X$ respecting the operations in $\{\cdot, \le, -\}$. Then there exists a representation $h^\circ$ in the same signature but representing the top element $\top$ as an equivalence relation. Furthermore, if $h$ respects the operations in $\{\land, \lor, \id, 0\}$ then so too does $h^\circ$.
\end{lemma}
\begin{proof} 
  The proof is largely the same as for Lemma~\ref{lemma:latticeorderedtop}, though we must recover faithfulness with a different approach. Recall that $0$ is forced to be represented as the empty set since  $0 \le -0$, and so any representation preserving $\{\cdot, \le -\}$ also trivially preserves $0$. Again we define $h^\circ \colon \MA \to \wp(X \times X)$ as $h^\circ \colon a \mapsto h(a) \cap h(T)^\circ$ and take distinct $a,b \in M(A)$ with the assumption that $b \nleq a$. Then there exists a nonzero $c$ such that $c \le b$ and $c \le -a$. Hence we can distinguish between $a$ and $b$ if we witness a nonempty $h^\circ(c)$.
  
  Since $c$ is nonzero we use Lemma~\ref{lemma:idempotentlemma} to conclude that $\dom(c)$ fixes a point $x \in X$ under $h$. Now we note that, since $\top$ has maximum domain and range and composition on the right cannot restrict domain, $D(c) \le D(c)T = cT$. We must witness this composition as in Figure~\ref{fig:witnesscT} and so $h^\circ(D(c)) \le h^\circ(cT) = h^\circ(c)h^\circ(T)$. That is, $h^\circ(c)$ is nonempty.
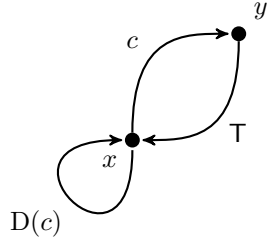
\begin{figure}[ht]
\centering
\begin{tikzpicture}[->,>=stealth',shorten >=1pt,auto,node distance=3cm,thick]]
\tikzstyle{vertex}=[circle, fill=black, draw=black, inner sep = 0.06cm]

\node[vertex, label = below left: $x$] (x) at (0:0cm) {};
\node[vertex, label = above right: $y$] (y) at (45:2cm) {};
\draw[->] (x) .. controls ([xshift=0mm, yshift=1.44cm] x) and ([xshift=-7.22mm, yshift=0cm] y)  .. (y);
\draw[->] (y) .. controls ([xshift=0mm, yshift=-1.44cm] y) and ([xshift=7.22mm, yshift=0cm] x)  .. (x);
\Loop[dist=2cm,dir=SOWE,label=$\dom(c)$,labelstyle=below left](x);
\node at (1.4cm,0.1cm) {$\top$};
\node at (0cm,1.3cm) {$c$};
\end{tikzpicture}
\caption{Witnessing the composition $D(c) \le cT$.}\label{fig:witnesscT}
\end{figure}
\end{proof}

  We noted before that the definition of complementation requires care in the absence of a top element. We used here relative complementation which mimics the definition of complementation when $\top$ is present: that if $x$ is related to $y$ by an element of the algebra, then for all relations $a$ we have that $(x,y)$ belongs to just one of $\{-a,a\}$. Under this weaker definition and without $\top$ acting as the universal relation we could have, for example, a situation as in Figure~\ref{fig:anotfunction} such that no element relates $x$ to $z$ or $z$ to $x$. If this occurs, we cannot take the complement to reason that $\id$ acts as the universal relation on these points, as we did in Lemma~\ref{lemma:orderedcomplemented}. Under the weaker definition of relative complementation, these proofs require that an element already relates these two points. 

  Alternatively, we can represent complementation as universal complementation in which the complement is taken with respect to $\wp(X \times X)$, where $X$ is the base set of the representation. That is, in the absence of a top element we can take complements with respect to a universal relation. Under this interpretation it will turn out that if $\top$ does exist then it must act as an equivalence relation in any representation. We thank Marcel Jackson for the following observation.

\begin{lemma}\label{lemma:strongcomplement}
  Let $\S$ be a complemented semigroup of binary relations with complement taken with respect to a universal relation. If there exists an idempotent $f$ such that $f(-f) = -f = (-f)f$ and $-f$ is also idempotent, then $f$ is the universal relation.
\end{lemma}
\begin{proof}
  Suppose $f$ relates $x$ to $y$ and, for contradiction, $f$ does not relate $y$ to $x$. Then $y(-f)x$ and, by assumption, $x(-f)x$. Then composing we get that $x(-f)y$, a contradiction. Hence, $f$ is a symmetric, and so a reflexive binary relation.
  
  We now know that $f$ is an equivalence relation on its domain. Suppose that the domain of $f$ is not full, and so we have that $x(-f)y$ and $y(-f)x$. Since $-f$ is idempotent we compose to get $x(-f)x$. If $x$ is in the domain of $f$ then we have a contradiction. If not, take $y$ to be in the domain of $f$ to reach a similar contradiction. Hence, $f$ is an equivalence relation with full domain.
\end{proof}

  Hence, any representation of $\MA$ respecting the operations in $\{\cdot, \le, -\}$ in which $-$ is represented as universal relation will always represent $\top$ as that universal relation. This extends the undecidability of representability of finite algebras in this signature to include infinite representations.
  
  It would be interesting to see if we can do the same for infinite representability of lattice-ordered semigroups.
  
\begin{problem}
  We know that finite representability is undecidable for lattice-ordered semigroups. Can the same be said of representability in general?
\end{problem}
  
  Another problem to consider is semigroups with either form of complementation but no order. As far as we can determine this problem remains unexplored in the literature, and is mentioned by Schein~\cite{Schein1991}.
  
\begin{problem}
  Is representability or finite representability decidable in the signature $\langle \cdot, - \rangle$, with either relative or universal complementation? Are representable algebras in this signature finitely axiomatisable?
\end{problem}


\bibliography{MNeuzerling}

\end{document}